\documentclass{amsart}
\usepackage{amssymb}
\usepackage{url}

\newcommand{\ZZ}{\mathbf{Z}}
\newcommand{\divides}{\mathbin|}

\newtheorem{thm}{Theorem}
\newtheorem{prop}[thm]{Proposition}

\newtheorem{lem}[thm]{Lemma}
\theoremstyle{remark}
\newtheorem{rem}{Remark}

\begin{document}

\title{A subquadratic algorithm for computing the $n$-th Bernoulli number}
%\date{28th August 2012}
\author{David Harvey}
\address{School of Mathematics and Statistics, University of New South Wales, Sydney NSW 2052, Australia}
\email{d.harvey@unsw.edu.au}
\urladdr{http://web.maths.unsw.edu.au/~davidharvey/}

\begin{abstract}
We describe a new algorithm that computes the $n$th Bernoulli number in $n^{4/3 + o(1)}$ bit operations. This improves on previous algorithms that had complexity $n^{2 + o(1)}$.
\end{abstract}

\maketitle

\section{Introduction}
\label{sec:introduction}

The Bernoulli numbers $B_0, B_1, B_2, \ldots$ are rational numbers defined by
 \[  \frac{t}{e^t - 1} = \sum_{k=0}^\infty \frac{B_k}{k!} t^k. \]
Every odd-index Bernoulli number is zero except for $B_1 = -1/2$. The von Staudt--Clausen theorem states that the denominator of $B_{2n}$ is precisely the product of the primes $p$ such that $p-1 \divides 2n$, and Euler's formula
\begin{equation}
\label{eq:euler}
 B_{2n} = (-1)^{n+1} \frac{2(2n)!}{(2\pi)^{2n}} \zeta(2n),
\end{equation}
where $\zeta(s) = \sum_{k=1}^\infty k^{-s}$ is the Riemann zeta function, implies that the number of bits in the numerator of $B_{2n}$ is $\Theta(n \log n)$.

It is known that the first $n$ Bernoulli numbers may be computed simultaneously in $n^2 \log^{2+o(1)} n$ bit operations \cite{BH-bernoulli}. This bound is optimal up to logarithmic factors, as the total number of bits being computed is $\Theta(n^2 \log n)$. (In this paper, ``bit operations'' always means number of operations in the multitape Turing model, as in \cite{Pap-complexity}.)

The situation concerning computation of a \emph{single} $B_n$ is less satisfactory. As $B_n$ has $n^{1+o(1)}$ bits, it is conceivable that it can be computed in only $n^{1+o(1)}$ bit operations. However, the best published complexity bounds have the shape $n^{2 + o(1)}$, which is essentially no better than computing all of $B_0, \ldots, B_n$. This is achieved by two quite different algorithms: the ``zeta function algorithm'', which approximates $\zeta(2n)$ in the right hand side of \eqref{eq:euler} via the Euler product (this has been rediscovered numerous times --- see the discussion in \cite{Har-bernmm}), and the multimodular algorithm introduced by the author in \cite{Har-bernmm}.

In this paper we close two thirds of this gap. Our main result is:
\begin{thm}
\label{thm:main}
Let $\frac13 \leq \alpha \leq \frac12$. The Bernoulli number $B_n$ may be computed in
 \[ n^{1+\alpha} \log^{4-\alpha + o(1)} n \]
bit operations, using
 \[ O(n^{2-2\alpha} \log^{1 + 2\alpha} n) \]
bits of space.
\end{thm}
The bounds are uniform in $\alpha$, i.e.~the implied $O(\cdot)$ constant does not depend on $\alpha$, and the $o(1)$ term approaches zero as $n \to \infty$, independently of $\alpha$.

In particular, taking $\alpha = 1/3$, we obtain the time bound
 \[ n^{4/3} \log^{11/3+o(1)} n \]
and space bound $O(n^{4/3} \log^{5/3} n)$. The parameter $\alpha$ permits a time-space tradeoff. Taking $\alpha = 1/2$ increases the time to $n^{3/2} \log^{7/2 + o(1)}$, but reduces the space usage to $O(n \log^2 n)$. In this latter case, further logarithmic savings in time and space may be achieved; see Remark \ref{rem:1/2}.

Our strategy may be summarised as follows. It is technically convenient to work with the Genocchi numbers \cite[p.~49]{Com-combinatorics} given by
 \[ G_n = 2(1 - 2^n)B_n. \]
It follows immediately from the von Staudt--Clausen theorem and Fermat's little theorem that $G_n \in \ZZ$. Moreover $\log G_n = O(n \log n)$, so asymptotically the number of bits we must determine is the same as for $B_n$.

In \cite{Har-bernmm}, we gave a formula, sometimes known as a Voronoi congruence, that expresses $B_n \pmod p$ as a sum of $O(p)$ terms, for a prime $p$. Evaluating this formula for sufficiently many $p$, and combining the results using the Chinese remainder theorem, led to the overall complexity bound $n^{2 + o(1)}$ for computing $B_n$.

Proposition \ref{prop:congruence} below may be interpreted as a generalisation of this formula from a congruence modulo $p$ to a congruence modulo $p^s$, expressing $G_n \pmod{p^s}$ as a sum of $O(ps)$ terms. Evaluating this formula in the straightforward way has complexity $O(p s^2)$ (ignoring logarithmic factors), because we are doing arithmetic in $\ZZ/p^s\ZZ$, whose elements have $O(s \log p)$ bits. This approach has more flexibility than that of \cite{Har-bernmm}, as we may choose $s$ as a function of $p$ to optimise the total cost. Unfortunately, it turns out that this still leads to a quasi-quadratic complexity bound for computing $B_n$.

However, we make two key observations. First, provided $p$ is not too small compared to $s$, we can use techniques of fast polynomial arithmetic to save a factor of roughly $s$ in the evaluation of the formula of Proposition \ref{prop:congruence}. This trick is already enough to lower the overall cost to $n^{3/2 + o(1)}$ (see Remark \ref{rem:1/2}). Second, by some algebraic rearrangement and careful choice of coefficient rings, we may evaluate the formula of Proposition \ref{prop:congruence} for many primes simultaneously. This reduces the complexity further to $n^{4/3 + o(1)}$.

We do not know for what $n$ an efficient implementation of these new algorithms would be faster than existing implementations of the quasi-quadratic algorithms. This is an interesting question for further study.

\section{Congruences}

\begin{prop}
\label{prop:congruence}
Let $n \geq s \geq 1$ and let $p$ be an odd prime. Let
 \[ F_p(x) = \sum_{k=0}^{s-1} \binom{n}{k+1} G_{k+1} p^k x^{s-1-k} \in \ZZ[x]. \]
Then
 \[ G_n = \sum_{j=0}^{p-1} (-1)^j j^{n-s} F_p(j) \pmod{p^s}. \]
\end{prop}
\begin{proof}
Our proof is modelled on \cite[Prop.~9.1.3]{Coh-nt-v2}.

The exponential generating function for the $G_n$ is
 \[ \sum_{n \geq 0} \frac{G_n}{n!} t^n = \sum_{n \geq 0} \frac{2 B_n}{n!} t^n - \sum_{n \geq 0} \frac{2 B_n}{n!} (2t)^n = \frac{2t}{e^t - 1} - \frac{4t}{e^{2t} - 1} = \frac{2t}{e^t + 1}. \]
The Genocchi polynomials
 \[ G_n(x) = \sum_{k=0}^n \binom{n}{k} G_k x^{n-k} \in \ZZ[x] \]
have exponential generating function given by
 \[ E(t, x) = \sum_{n \geq 0} \frac{G_n}{n!} t^n = \left(\sum_{k \geq 0} \frac{G_k}{k!} t^k\right) \left(\sum_{m \geq 0} \frac{x^m}{m!} t^m\right) = \frac{2t e^{tx}}{e^t + 1}. \]
Now on one hand we have
 \[ \sum_{j=0}^{p-1} (-1)^j E(pt, j/p) = \sum_{n \geq 0} \frac{p^n}{n!} \sum_{j=0}^{p-1} (-1)^j G_n(j/p) t^n, \]
while on the other hand this sum is also equal to
 \[ \sum_{j=0}^{p-1} (-1)^j \frac{2pte^{tj}}{e^{pt} + 1} = \frac{2pt}{e^{pt} + 1} \sum_{j=0}^{p-1} (-e^t)^j = \frac{2pt}{e^t + 1} = p \sum_{n \geq 0} \frac{G_n}{n!} t^n. \]
Equating coefficients of $t^n$ and using $G_0 = 0$ we obtain
 \[ G_n = p^{n-1} \sum_{j=0}^{p-1} (-1)^j G_n(j/p) = \sum_{j=0}^{p-1} (-1)^j \sum_{k=0}^{n-1} \binom{n}{k+1} G_{k+1} j^{n-k-1} p^k. \]
Truncating this sum modulo $p^s$ yields the desired congruence.
\end{proof}

\section{Algorithms}

Let $n \geq s \geq 1$, and define
 \[ F(x) = \sum_{k=0}^{s-1} \binom{n}{k+1} G_{k+1} x^{s-1-k} \in \ZZ[x]. \]
Note that $F(x)$ depends on $n$ and $s$, but (crucially) not on $p$. For any prime $p \geq 3$ and any $0 \leq j < p$ we have $F_p(j) = p^{s-1} F(j/p)$. We obtain the following bounds for the coefficients of $F(x)$ and for $F_p(j)$.
\begin{lem}
Let $n \geq 1$ and $0 \leq k < n$. Then
 \[ \left| \binom{n}{k+1} G_{k+1}\right| \leq 7 (n/\pi)^{k+1}. \]
\end{lem}
\begin{proof}
For $k = 0$ the assertion is that $n \leq 7(n/\pi)$. For even $k \geq 2$, we have $G_{k+1} = 0$. For odd $1 \leq k < n$, by \eqref{eq:euler} we have
\begin{align*}
 \left| \binom{n}{k+1} G_{k+1}\right| & = \frac{n!}{(k+1)!(n-k-1)!} 2 (2^{k+1} - 1) \frac{2 (k+1)!}{(2\pi)^{k+1}} \zeta(k+1)  \\
 & \leq 4 \zeta(2) \frac{n!}{(n-k-1)!} \frac{1}{\pi^{k+1}} \leq 6.579... (n/\pi)^{k+1}. \qedhere
\end{align*}

\end{proof}
\begin{lem}
\label{lem:Fp-bound}
Let $n \geq s \geq 4$. Let $p$ be an odd prime and let $0 \leq j < p$. Then
 \[ |F_p(j)| \leq 3 (np/\pi)^{s+1}. \]
\end{lem}
\begin{proof}
By the previous lemma we have
\begin{align*}
 |F_p(j)| & \leq \sum_{k=0}^{s-1} 7 (n/\pi)^{k+1} p^k j^{s-1-k}
            \leq 7 p^{s-1} \sum_{k=0}^{s-1} (n/\pi)^{k+1} = 7 p^{s-1} (n/\pi) \frac{(n/\pi)^s - 1}{n/\pi - 1} \\
          & \leq \frac{7}{3^2(4/\pi - 1)} (np/\pi)^{s+1} = 2.846... (np/\pi)^{s+1}. \qedhere
\end{align*}
\end{proof}

We recall some standard results concerning the complexity of integer and polynomial arithmetic; all of this may be found in \cite{vzGG-compalg}.

Let $R = \ZZ/2^M \ZZ$ where $M \geq 1$. Addition and subtraction in $R$ require $O(M)$ bit operations. Multiplication in $R$ costs $M \log^{1+o(1)} M$ bit operations, using $O(M)$ bits of space, via FFT methods. Division in $R$ (where possible) has the same asymptotic time and space complexity as multiplication, using Newton's method. If $G \in R[x]$ is a polynomial of degree $s$, and $x_1, \ldots, x_t \in R$, with $t \leq s$, then we may simultaneously evaluate $G(x_1), \ldots, G(x_t) \in R$ using a fast multipoint evaluation algorithm in $s M \log^{1+o(1)} (s M) \log s$ bit operations. The simplest such algorithms have space complexity $O(s M \log s)$, but this can be reduced to $O(s M)$ by the method of \cite[Lemma 2.1]{vzGS-frobenius}.

Now let $p$ be an odd prime, $s \geq 1$, and $R = \ZZ/p^s\ZZ$. We assume here that $p \leq n$ and $s \leq n$. The results are similar: addition and subtraction in $R$ require $O(s \log p) = O(s \log n)$ bit operations, and multiplication in $R$ costs $s \log^{2+o(1)} n$ bit operations.

Finally we mention that the primes $p \leq N$ may be enumerated by a straightforward sieve method in $N^{1+o(1)}$ bit operations.

\begin{prop}
\label{prop:padic-bound}
Let $n \geq s \geq 4$ and let $N \leq n$. Let $P$ be a set of primes with $3 \leq p < N$ for all $p \in P$. Assume that $\sum_{p \in P} p \leq s$. Then the residues $G_n \pmod{p^s}$ may be computed for all $p \in P$ simultaneously in
 \[ s^2 \log^{3+o(1)} n \]
bit operations, using
 \[ O(s^2 \log n) \]
bits of space.
\end{prop}
\begin{proof}
Let $M = \lceil \log_2(3 (nN/\pi)^{s+1}) \rceil + 1$. For this choice of $M$, by Lemma \ref{lem:Fp-bound} we have $|F_p(j)| < 2^M / 2$ for all $p \in P$, $0 \leq j < p$, so to compute $F_p(j)$ it suffices to determine it modulo $2^M$. Note that $M = O(s \log n)$.

We perform the following steps, each of which uses $O(s^2 \log n)$ space.

{\it Step 1.} Compute $G_k$ for $1 \leq k \leq s$ using (for example) the algorithm of \cite{BH-bernoulli}. This costs $s^2 \log^{2+o(1)} s = s^2 \log^{2+o(1)} n$ bit operations.

{\it Step 2.} Compute $\binom{n}{k}$ for $1 \leq k \leq s$. Using a straightforward algorithm this can be done in $O(s^2 \log^2 n)$ bit operations.

{\it Step 3.} Compute the coefficients of $F(x)$, by computing the products $\binom{n}{k} G_k$ for $0 \leq k \leq s$. Each product needs $s \log^{2+o(1)} n$ bit operations. The total cost is $s^2 \log^{2+o(1)} n$ bit operations.

{\it Step 4.} Compute $j/p \pmod{2^M}$ for each $p \in P$, $0 \leq j < p$. Each division costs $M \log^{1+o(1)} M = s \log^{2+o(1)} n$ bit operations. Since we have assumed that $\sum_{p \in P} p \leq s$, the total cost is $s^2 \log^{2+o(1)} n$ bit operations.

{\it Step 5.} Regarding $F(x)$ as a polynomial in $(\ZZ/2^M\ZZ)[x]$, evaluate simultaneously $F(j/p) \pmod{2^M}$ for all $p \in P$, $0 \leq j < p$. This costs $s^2 \log^{3+o(1)} n$ bit operations.

{\it Step 6.} For each $p \in P$, $0 \leq j < p$, recover $F_p(j) = p^{s-1} F(j/p) \pmod{2^M}$, and hence the exact integer $F_p(j)$. Since $p^{s-1} \leq 2^M$, we may compute $p^{s-1}$, and then $F_p(j)$, in $M \log^{1+o(1)} M = s \log^{2+o(1)} n$ bit operations, and thus the total cost is $s^2 \log^{2+o(1)} n$ bit operations.

{\it Step 7.} For each $p \in P$, $0 \leq j < p$, compute $j^{n-s} \pmod{p^s}$. Each power costs $(\log n) (s \log^{2+o(1)} n) = s \log^{3+o(1)} n$ bit operations, so the total cost is $s^2 \log^{3+o(1)} n$ bit operations.

{\it Step 8.} Use Proposition \ref{prop:congruence} to recover $G_n \pmod{p^s}$ for each $p \in P$. The cost is $s^2 \log^{2+o(1)} n$ bit operations.
\end{proof}

\begin{rem}
The complexity of step 7 can be improved, by computing first $q^{n-s} \pmod{p^s}$ for \emph{primes} $q < p$, and then using $(j_1 j_2)^{n-s} = j_1^{n-s} j_2^{n-s}$ for composite $j = j_1 j_2$. This saves a factor of $\log n$ in this step, provided that $p$ is not too small, say $p > n^c$ for any fixed $c > 0$. This will be the case for almost all primes $p$ used in the proof of Theorem \ref{thm:main}.
\end{rem}

\begin{rem}
In a practical setting, one may wish to replace the ring $\ZZ/2^M\ZZ$ by $\ZZ/T\ZZ$ where $T$ is a suitably large integer not divisible by any $p \in P$. For example, one could take $T$ to be a product of many word-sized primes $q$ for which there exist efficient number-theoretic transforms modulo $q$. Under this scheme, the expensive evaluation in Step 5 could be performed for each $q$ separately, and then the $F_p(j)$ could be reconstructed in Step 6 using the Chinese remainder theorem. This approach does not change the asymptotic complexity, but potentially yields a drastic improvement in memory locality.
\end{rem}

\begin{rem}
Further practical savings may be realised by using the easily-proved fact that $G_n(1 - x) = -G_n(x)$ for even $n$, so that
 \[ G_n = 2\sum_{j=1}^{(p-1)/2} (-1)^j j^{n-s} F_p(j) \pmod{p^s}. \]
Coupled with the observation that essentially half of the coefficients of $F(x)$ are zero, this leads to a savings of a factor of two in the main evaluation step.
\end{rem}

Now we may prove the main result.
\begin{proof}[Proof of Theorem \ref{thm:main}]
Recall that $1/3 \leq \alpha \leq 1/2$. We will take
 \[ N = \lfloor n^{\alpha} \log^{1 - \alpha} n \rfloor, \qquad s = \lfloor 2 n^{1-\alpha} \log^\alpha n \rfloor. \]
We may assume that $n$ is large enough so that $n \geq s \geq N \geq 4$. In particular we may assume that the hypotheses of Proposition \ref{prop:padic-bound} are satisfied.

Let $P$ be the set of odd primes $p < N$, so that $|P| = O(N/\log N) = O(n^\alpha \log^{-\alpha} n)$. Let $r = \lfloor 2 n^{1-2\alpha} \log^{2\alpha - 1} n \rfloor$. Note that $r \geq 1$ for sufficiently large $n$. Partition $P$ into $d$ sets $P_1, \ldots, P_d$ of cardinality at most $r$, where $d = O(|P|/r) = O(n^{3\alpha-1}\log^{1-3\alpha} n)$. For each $i$ we have $\sum_{p \in P_i} p \leq |P_i|N \leq rN \leq s$.

Apply Proposition \ref{prop:padic-bound} to each set $P_i$ separately. The space usage for each invocation is $O(s^2 \log n) = O(n^{2 - 2\alpha} \log^{2\alpha + 1} n)$. This space may be reused for each $P_i$. The total time cost is $d s^2 \log^{3+o(1)} n = n^{1+\alpha} \log^{4 - \alpha + o(1)} n$.

At this stage we have computed $G_n \pmod{p^s}$ for all $p \in P$. This is enough to determine $G_n$ (for sufficiently large $n$), because
 \[ \log \prod_{p \in P} p^s = s \sum_{3 \leq p < N} \log p \sim s N = 2 n \log n + O(n), \]
whereas $\log G_n = n \log n + O(n)$. Using fast Chinese remaindering we may then recover $G_n$, and hence $B_n$, in $n^{1+o(1)}$ bit operations.
\end{proof}

\begin{rem}
\label{rem:1/2}
We sketch an algorithm that improves the time and space complexities to respectively $n^{3/2} \log^{3 + o(1)} n$ and $O(n \log n)$ in the case $\alpha = 1/2$. Consider the algorithm of Proposition \ref{prop:padic-bound} applied to a set $P = \{p\}$ consisting of a single prime. The evaluation points $j/p$, for $0 \leq j < p$, now form an arithmetic progression. We relax the condition $p \leq s$, instead allowing $p$ as large as $s \log s$. Instead of evaluating at all $p$ points simultaneously, we first evaluate at only $s$ points, and then use the value-shifting algorithm of \cite[Theorem~3.1]{Sho-small} (alternatively the algorithm of \cite[Theorem~5]{BGS-recurrences}) to evaluate at the remaining $p - s$ points, in blocks of $s$ points at a time. Then in the proof of Theorem \ref{thm:main}, we take $s = \lfloor 2n^{1/2} \rfloor$ and $N = \lfloor n^{1/2} \log n \rfloor$, and only use Proposition \ref{prop:padic-bound} for one prime at a time. This leads to the complexity bounds stated above; we omit the proof, which is similar to that of Theorem \ref{thm:main}.
\end{rem}

%\section*{Acknowledgments}

{\it Acknowledgments.} Many thanks to Joe Buhler, Bernd Kellner and an anonymous referee for their comments on a draft of this paper. The author was partially supported by the Australian Research Council, DECRA Grant DE120101293.

\bibliographystyle{amsalpha}
\bibliography{bernpadic}

\end{document}